\documentclass[10pt,electronic]{amsart}        
\usepackage[applemac]{inputenc}
\usepackage[T1]{fontenc}
\usepackage[small,it]{caption}
\usepackage[english]{babel}
\usepackage{url}
\usepackage[hmargin=3.7cm,vmargin=3.5cm]{geometry} 
\usepackage{graphicx}                         
\usepackage{amsmath}
\usepackage{amsthm}                      
\usepackage{amssymb}
\usepackage{amscd}     
\usepackage{mathrsfs}
\usepackage{stmaryrd}
\usepackage{enumerate}

\usepackage[isbn=false,doi=false,eprint=false,url=false,style=authoryear,dashed=false,backend=bibtex]{biblatex}
\makeatletter
\newrobustcmd*{\parentexttrack}[1]{%
  \begingroup
  \blx@blxinit
  \blx@setsfcodes
  \blx@bibopenparen#1\blx@bibcloseparen
  \endgroup}
\AtEveryCite{%
  }
\makeatother
\renewcommand{\cite}{\parencite}
\DeclareNameAlias{sortname}{last-first}
\renewbibmacro*{volume+number+eid}{%
  \printfield{volume}%
  \setunit*{\addnbspace}
  \printfield{number}%
  \setunit{\addcomma\space}%
  \printfield{eid}}
 \DeclareFieldFormat[article]{volume}{\textbf{#1}}
\DeclareFieldFormat[article]{number}{{\mkbibparens{#1}},}
\DeclareFieldFormat{pages}{#1 pages}
\DeclareFieldFormat
  [article,inbook,incollection,inproceedings,patent,thesis,unpublished]
  {pages}{#1}
\renewbibmacro{in:}{}
\theoremstyle{plain}                                                           
\newtheorem{thm}{Theorem}[section]
\newtheorem{lem}[thm]{Lemma}

\theoremstyle{definition}

\newcommand{\field}[1]{\ensuremath{\mathbf{#1}}}

\newcommand{\Z}{\ensuremath{\field{Z}}}

\newcommand{\A}{\mathcal A}
\newcommand{\M}{\mathcal{M}}
\newcommand{\MM}{\overline{\mathcal{M}}}

\newcommand{\Bl}{\mathrm{Bl}}

\title{The Chow ring of a Fulton--MacPherson compactification}

\author{Dan Petersen}
\thanks{The author is supported by the Danish National Research Foundation through the
Centre for Symmetry and Deformation (DNRF92).}
\email{danpete@math.ku.dk}
\address{Institut for Matematiske Fag \\
K{\o}benhavns Universitet \\
Universitetsparken 5 \\
2100 K{\o}benhavn {\O} }
\begin{document} 
 \maketitle   
 
 \begin{abstract}
We give a short proof of a presentation of the Chow ring of the Fulton--MacPherson compactification of $n$ points on an algebraic variety. The result can be found already in Fulton and MacPherson's original paper. However, there is an error in one of the lemmas used in their proof. In the process we also determine the Chow rings of weighted Fulton--MacPherson compactifications.
 \end{abstract}
 
\section{Introduction} 

For a topological space $X$, let $F(X,n)$ denote the configuration space of $n$ distinct ordered points on $X$. Spaces of this form have been intensely studied for a long time, going back at least to the calculation of the cohomology ring of $F(\mathbf R^2,n)$ \cite{arnoldbraid}, which is a classifying space for the pure braid group on $n$ strands. 

The seminal paper \cite{fmcompactification} studied the question of how one can \emph{compactify} the space $F(X,n)$, in the special case that $X$ is a smooth projective algebraic variety. This question may at first seem absurd: what could be nicer than the obvious inclusion $F(X,n) \hookrightarrow X^n$? However, in algebraic geometry one often wishes to compactify an open variety in such a way that it becomes the complement of a divisor with normal crossings. To this end, they proposed a different compactification denoted $X[n]$, now called the \emph{Fulton--MacPherson compactification}. As an application, they could use the work of \cite{morgan} to write down an explicit model of $F(X,n)$ in the sense of rational homotopy theory. Their model was later simplified  (independently) by \cite{krizconfiguration} and \cite{totaro}. 

Just like $X^n$, the space $X[n]$ admits a modular interpretation, where the boundary  para\-met\-rizes certain `degenerate' configurations of points on $X$. However, instead of allowing points to collide, the space $X[n]$ is set up so that the variety $X$ itself is allowed to degenerate in a controlled manner. The effect is that when points try to come together, $X$ acquires a new irreducible component --- a projective space of the appropriate dimension --- on which the points end up and remain distinct. See \cite[194--195]{fmcompactification} for a more precise description. They show that the boundary $X[n] \setminus F(X,n)$ will indeed be a strict normal crossing divisor, that the combinatorial structure of the boundary strata admits a pleasant combinatorial description in terms of rooted trees, and that $X[n]$ can be constructed from $X^n$ by an explicit sequence of blow-ups in smooth centers. 

Their construction is related to (and was inspired by) the Deligne--Mumford compactification $\M_{g,n} \subset \MM_{g,n}$ of the moduli space of smooth curves of genus $g$ with $n$ distinct ordered points. In fact, the fiber of $\M_{g,n} \to \M_g$ over a moduli point $[X]$ is the configuration space $F(X,n)$, and the fiber of $\MM_{g,n} \to \MM_g$ over the same point is the Fulton--MacPherson compactification $X[n]$. 

A real version of the space $X[n]$ which is defined for any manifold $X$ was independently discovered by Kontsevich in the context of knot invariants and perturbative Chern--Simons theory \cite{kontsevichfeynman,axelrodsinger}. It can be constructed by an identical construction as $X[n]$, replacing algebro-geometric blow-ups with real analytic blow-ups. For this version, $X[n]$ becomes a manifold-with-corners, and $F(X,n) \hookrightarrow X[n]$ a homotopy equivalence.
Particularly useful is the case $X=\mathbf R^d$, in which case $X[n]$ is a version of the little $d$-cubes operad.

One of the results of \cite{fmcompactification} is the calculation of the Chow ring of $X[n]$, considered as an algebra over the Chow ring of $X^n$. This is the main theorem of Section 5 of their paper. Given that $X[n]$ is constructed from $X^n$ by an explicit sequence of blow-ups, one might expect this to be rather straightforward, but some care is needed to manage the combinatorics involved and to avoid redundant relations. Unfortunately, there is a gap in their proof. To carry out the calculation they repeatedly apply a number of lemmas about Chow rings of blow-ups; one of these, Lemma 5.4, is incorrect. We comment more on this  
Section \ref{sectionfalse}. 
One reason for writing this note is that one can find many calculations of Chow rings of iterated blow-ups modeled on Fulton and MacPherson's  in the literature, in particular making use of Lemma 5.4.



The main result of this note is a different proof of the presentation of the Chow ring of $X[n]$. It is plausible that the original proof could be modified to work using a corrected version of Lemma 5.4 (cf.\ the footnote on p.~4), but we choose instead to take a slightly different approach which sidesteps this lemma completely. Moreover, in the process we compute presentations of the Chow rings of \emph{weighted} Fulton--MacPherson compactifications $X_\A[n]$ for any weights $\A$. The space $X_\A[n]$ was introduced by \cite{routis} by analogy with the moduli space $\MM_{g,\A}$ of weighted pointed stable curves \cite{hassettweighted}: the space $X_\A[n]$ bears the same relationship to the space $\MM_{g,\A}$ as $X[n]$ does to $\MM_{g,n}$. This presentation of the Chow ring of $X_\A[n]$ was previously given in the same paper of Routis. 

Specifically, it is a consequence of general results from \cite{wonderfulcompactification} that there are many possible ways one can construct $X[n]$ from $X^n$ by blow-ups, corresponding to different orderings of the blow-up loci. We choose to work with a different inductive construction than Fulton and MacPherson, which has the advantage that each intermediate step \emph{and} each blow-up center is itself a weighted Fulton--MacPherson compactification. This leads to a very short inductive argument, which we carry out in Section \ref{sectionproof}.  

\subsection{Conventions.} We denote by $[n]$ the set of integers $\{1,\ldots,n\}$. Following Fulton and MacPherson we say that $S, T \subseteq [n]$ \emph{overlap} if $S \cap T \not\in \{\emptyset,S,T\}$. If $X$ is a smooth algebraic variety, then $A^\bullet(X)$ denotes its Chow ring with integer coefficients. However, the arguments would work equally well for the cohomology ring, in any cohomology theory where Lemmas \ref{keel} and \ref{blowuplemma} remain valid (i.e.\ where standard properties of blow-ups are satisfied). Like Fulton and MacPherson, we use throughout the language of varieties over algebraically closed fields, even though the results remain valid also for the \emph{relative} Fulton--MacPherson compactification for a smooth family $X \to S$ of varieties over a given nonsingular variety $S$.

\subsection{Acknowledgements.} The error in Lemma 5.4 was pointed out to me in a referee report for the paper \cite{universalcurverationaltails}. I am grateful to the anonymous referee for their remarkably careful reading.

\section{Lemma 5.4 in Fulton--MacPherson} \label{sectionfalse}
 
We recall the standing assumptions of \cite[Section 5]{fmcompactification}
 : $Z$ is a closed subvariety of $Y$; both are smooth and irreducible; $A^\bullet(Y)\to A^\bullet(Z)$ is surjective; $\widetilde Y$ denotes $\Bl_ZY$; $J_{Z/Y}$ denotes the kernel of $A^\bullet(Y)\to A^\bullet(Z)$; finally, $P_{Z/Y}(t)$ denotes a Chern polynomial of $Z$ in $Y$, i.e. a polynomial
$$ P_{Z/Y}(t) = t^d + a_1t^{d-1} + \ldots + a_d \in A^\bullet(Y)[t]$$
where $d$ is the codimension of $Z$, $a_d = [Z]$, and $a_i$ for $0<i<d$ denotes any class in $A^i(Y)$ whose restriction to $A^i(Z)$ is the $i$th Chern class of the normal bundle of $Z$. The surjectivity hypothesis implies that Chern polynomials always exist. 

They first state the following lemma:

\begin{lem}[Lemma 5.3] \label{keel} $A^\bullet(\widetilde Y) = A^\bullet(Y)[E]/\langle J_{Z/Y}\cdot E, P_{Z/Y}(-E)\rangle$.
\end{lem}

We remark that this formula is valid also if $Z =\emptyset$, noting that $J_{Z/Y}$ will be all of $A^\bullet(Y)$ in this case. After this, they state the following:

\begin{lem}[Lemma 5.4]Assume that $V$ is another smooth irreducible subvariety with $A^\bullet(Y) \to A^\bullet(V)$ surjective, and that $V$ intersects $Z$ transversally. Then $A^\bullet(\widetilde Y) \to A^\bullet(\widetilde V)$ is surjective with kernel $ J_{V/Y}$ if $Z \cap V$ is nonempty, and $\langle J_{V/Y}, E \rangle$ if $Z \cap V$ is empty. 
\end{lem}

This is claimed to follow from Lemma 5.3. However, the conclusion depends on whether or not $Z \cap V$ is empty, and as noted here, Lemma 5.3 does not. In fact, applying Lemma 5.3 twice gives instead the following result.

\begin{lem}[Lemma 5.4, corrected] \label{correctedlemma} Assume that $V$ is another smooth irreducible subvariety with $A^\bullet(Y) \to A^\bullet(V)$ surjective, and that $V$ intersects $Z$ transversally. Then $A^\bullet(\widetilde Y) \to A^\bullet(\widetilde V)$ is surjective with kernel $\langle J_{V/Y}, J_{Z\cap V/Y} \cdot E \rangle$. 
\end{lem}

 Clearly, the conclusion of the old lemma will be valid (if $Z \cap V$ is nonempty) precisely when $J_{Z \cap V /Y} \cdot E$ lies in the ideal generated by $J_{V/Y}$ in $A^\bullet(\widetilde Y)$. In particular, Lemma 5.4 fails already when $Y = \mathbf P^3$, $V = \mathbf P^2$ and $Z = \mathbf P^1$, intersecting in a point: we have $A^\bullet(\widetilde Y) = \Z[h,E]/\langle h^4,h^2E, E^2 - 2hE + h^2\rangle$ and $J_{Z \cap V/Y} \cdot E = \langle h E \rangle$, which is not in the ideal generated by $J_{V/Y} = \langle h^3 \rangle$. 



\section{Calculation of the Chow ring}\label{sectionproof}

We now give a calculation of the Chow ring of $X[n]$. We do this in two steps. First we give a presentation of the Chow rings of \emph{weighted} Fulton--MacPherson compactifications $X_\A[n]$. These were determined previously\footnote{Routis's computation of the Chow rings of weighted Fulton--MacPherson compactifications is an adaptation of Fulton and MacPherson's original argument, and in particular his proof relies on their Lemma 5.4. Routis has informed me that his argument can be modified so that it uses only the corrected version of Lemma 5.4 (Lemma \ref{correctedlemma} in this note). Nevertheless, the alternative argument given here may be of independent interest.} in \cite{routis}. However, Routis's presentation of these Chow rings doesn't specialize to the one obtained by Fulton--MacPherson in the case when all weights are $1$; there are a number of excess relations in the presentation. We show that when all weights are $1$, these excess relations can in fact be omitted, recovering the original presentation of Fulton and MacPherson. (This will be easier than trying to get rid of excess relations in each step of the induction.)

\subsection{The weighted Fulton--MacPherson compactification}

Let $\A = (a_1,\ldots,a_n) \in [0,1]^n$ be a collection of weights, and fix a smooth variety $X$. Following Routis we consider the \emph{weighted} Fulton--MacPherson compactification $X_\A[n]$ of $n$ points on $X$. Roughly speaking, it is a variant of the usual Fulton--MacPherson compactification where a subset $S \subseteq [n]$ of the markings are allowed to coincide if and only if $\sum_{i \in S} a_i \leq 1$. The space $X_\A[n]$ can also be described as a wonderful compactification \cite{wonderfulcompactification}: for each $S$ such that $\sum_{i \in S} a_i > 1$, consider the diagonal $\Delta_S \subset X^n$. The collection of all these diagonals form a building set whose wonderful compactification is $X_\A[n]$. In the extreme cases $\A=(0,0,\ldots,0)$ resp.\ $\A = (1,1,\ldots,1)$ we recover the spaces $X^n$ and $X[n]$, respectively. See \cite{routis} for a more precise description of these spaces and their properties. 

We remark that we don't actually need the weights to carry out the construction, only the combinatorial information about which subsets of $[n]$ satisfy $\sum_{i \in S} a_i > 1$. We shall say that $S \subseteq [n]$ is \emph{large} if $\sum_{i \in S} a_i > 1$ and that $S$ is \emph{small} otherwise. The collection of small subsets can be any abstract simplicial complex with vertex set $[n]$, and $X_\A[n]$ depends only on this abstract simplicial complex. 

If $\A$ and $\A'$ are weights with $a_i \geq a_i'$ for all $i$, then there is a reduction morphism $X_\A[n] \to X_{\A'}[n]$. In terms of the modular interpretation of these spaces, it contracts all extraneous components that have total weight less than $1$ after reducing weights from $\A$ to $\A'$. The hyperplanes $H_S = \{ \sum_{i \in S} a_i = 1\}$ separate the cube $[0,1]^n$ into different chambers, and if $\A$ and $\A'$ are in the same chamber then $X_\A[n] \to X_{\A'}[n]$ is an isomorphism. If $\A$ and $\A'$ are in adjacent chambers, separated by the hyperplane $H_T$, then $X_\A[n]$ is the blow-up of $X_{\A'}[n]$ in the iterated strict transform $\widetilde \Delta_T \subset X_{\A'}[n]$ of $\Delta_T \subset X^n$. We call $\widetilde \Delta_T$ a \emph{coincidence set}; it consists of those configurations where the markings indexed by $T$ coincide with each other. We observe that $\widetilde \Delta_T \subset X_\A[n]$ is itself a weighted Fulton--MacPherson compactification, 
where all the points indexed by $T$ have been removed and we have instead added a point of weight $\sum_{i \in T} a_i$. 

We can thus construct $X_\A[n]$ inductively by starting with $X^n$ --- corresponding to the weight vector $(0,0,\ldots,0)$ --- and increasing the weights along a path from the origin to $\A$ in $[0,1]^n$, in such a way that we only intersect at most one of the hyperplanes $H_T$ at a given time. Equivalently, we start with the simplex on $n$ vertices as our abstract simplicial complex of `small' sets, and then we remove one maximal face of the complex at a time. At each step of this inductive procedure we are blowing up a weighted Fulton--MacPherson compactification in a locus which is also isomorphic to a weighted Fulton--MacPherson compactification.

\subsection{The inductive proof} For $S \subset [n]$ we let $\Delta_S$ be the corresponding diagonal in $X^n$, and $J_S = \mathrm{ker}(A^\bullet(X^n) \to A^\bullet(\Delta_S))$. Let $c_S(t)$ denote a Chern polynomial for $\Delta_S$ in $X^n$. Specifically, if $S = \{i_1,\ldots,i_k\}$, then we may set $c_S(t) = \prod_{j=1}^{k-1} c_{i_ji_{j+1}}(t)$, where 
$$ c_{ij}(t) = \sum_{\ell=1}^{d}(-1)^\ell \mathrm{pr}_i^\ast (c_{d-\ell}(TX)) t^\ell + [\Delta_{ij}] $$
is a Chern polynomial for $\Delta_{ij} \subset X^n$. 

We will need the following lemma.
\begin{lem} \label{blowuplemma} Let as before $Y$ be a smooth variety, $Z$ a closed irreducible smooth subvariety, $V$ another smooth closed subvariety not contained in $Z$, and $\widetilde V \subset \widetilde Y = \Bl_Z Y$ the strict transform. Then:
\begin{enumerate}
\item If $Z$ and $V$ intersect transversally, then $P_{V/Y}(t)$ is a Chern polynomial for $\widetilde V \subset \widetilde Y$. 
\item If $Z$ is contained in $V$, then $P_{V/Y}(t - E)$ is a Chern polynomial for $\widetilde V \subset\widetilde Y$, where $E$ is the class of the exceptional divisor.  
\end{enumerate} \end{lem}
\begin{proof}\cite[Lemma 5.2]{fmcompactification}.\end{proof}

Together with Lemma \ref{keel} (Lemma 5.3) we are now in a position to give a presentation of the Chow rings of weighted Fulton--MacPherson compactification.

\begin{thm}[Routis] \label{thm1} Let $\A = (a_1,\ldots,a_n) \in [0,1]^n$ and let $X_\A[n]$ denote the corresponding weighted Fulton--MacPherson compactification. We have
$$ A^\bullet(X_\A[n]) = A^\bullet(X^n)[D_S]/\text{relations}$$
where there is a variable $D_S$ for all large subsets of $S$, and the relations are
\begin{enumerate}
\item $D_S \cdot D_T = 0$ if $S$ and $T$ overlap,
\item $J_S \cdot D_S = 0$, 
\item for each large subset $S$, $c_S(\sum_{S \subseteq V} D_V) = 0,$
\item if $S$ is large and $S'$ is arbitrary, and $\vert S \cap S' \vert = 1$, then 
$$ D_S \cdot c_{S'}(\sum_{S \cup S' \subseteq V} D_V) = 0.$$
\end{enumerate}
\end{thm}

\begin{thm}\label{thm2}
If $T$ is a small subset of $[n]$, then let $\widetilde \Delta_T$ denote the corresponding coincidence set in $X_\A[n]$. 

\begin{enumerate}[\rm (i)]
\item The ideal $J_{\widetilde \Delta_T/X_\A[n]}$ is generated by $J_T$, the elements $D_S$ for large $S$ that overlap $T$,
and for each  large set $S$ containing $T$, the element
$$ c_{S \setminus T \cup \{i\}}(\sum_{S \subseteq V} D_V),$$
where $i$ is an arbitrary element of $T$.
\item A Chern polynomial of $\widetilde \Delta_T$ is given by 
$$ c_T(-t + \sum_{\substack{T \subset V \\ V \text{ large}}} D_V).$$
\end{enumerate}
\end{thm}

\begin{proof}
We prove Theorems \ref{thm1} and \ref{thm2} simultaneously, by induction over $n$ and over the number of large subsets of $[n]$. 

To prove Theorem \ref{thm1} we write $X_\A[n]$ as a blow-up of $X_{\A'}[n]$ in a coincidence set $\widetilde \Delta_T$, where $\A'$ has one fewer large set than $\A$. By induction we know the Chow ring of $X_{\A'}[n]$, the ideal $J_{\widetilde \Delta_T/X_{\A'}[n]}$ and the Chern polynomial of $\widetilde \Delta_T$, so the Chow ring of $X_\A[n]$ is completely determined by Lemma \ref{keel} (Lemma 5.3). We get a new generator $D_T$ and one checks that the extra relations are exactly those predicted by Theorem \ref{thm1}. This finishes the proof.

For Theorem \ref{thm2}(i), we have noted that the coincidence set $\widetilde \Delta_T$ is again a weighted Fulton--MacPherson compactification, where all the points indexed by $T$ have been removed and we have instead added a point of weight $\sum_{i \in T} a_i$. By induction we therefore have a presentation of $A^\bullet(\widetilde \Delta_T)$ by generators and relations, using Theorem \ref{thm1}. The induced map $A^\bullet(X_\A[n]) \to A^\bullet(\widetilde \Delta_T)$ maps $A^\bullet(X^n)$ to $A^\bullet(\Delta_T) = A^\bullet(X^n)/J_T$, and it sends a generator $D_S$ to $0$ if $S$ and $T$ overlap, and to a corresponding generator in $A^\bullet(\Delta_T)$ otherwise. Theorem \ref{thm2}(i) then follows from Theorem \ref{thm1} and comparing relations.

Theorem \ref{thm2}(ii) is a direct consequence of Lemma \ref{blowuplemma}, noting that two coincidence sets either intersect transversally, or one is contained in the other \cite[Lemma 2.6]{wonderfulcompactification}. 
\end{proof}

Routis's result does not specialize to the original presentation of $A^\bullet(X[n])$ given by Fulton and MacPherson when $\A = (1,\ldots,1)$ because of the redundancies in the presentation. We now show how the presentation can be simplified in this case to obtain the original result. 

\begin{thm}[Fulton--MacPherson]\label{originalthm} Suppose that $\A = (1,1,\ldots,1)$. Then the presentation of $A^\bullet(X[n])$ can be simplified to $$ A^\bullet(X[n]) = A^\bullet(X^n)[D_S]/\text{relations}$$
where there is a variable $D_S$ for all $S \subseteq [n]$ with $\vert S \vert \geq 2$, and the relations are
\begin{enumerate}
\item $D_S \cdot D_T = 0$ if $S$ and $T$ overlap,
\item $J_S \cdot D_S = 0$, 
\item for any $i \neq j$, $c_{ij}(\sum_{i,j \in V} D_V) = 0.$
\end{enumerate} \end{thm}

\begin{proof}We argue first that the given relations imply $c_S(\sum_{S \subseteq V} D_V) = 0$, by induction over $\vert S \vert$ with base case $\vert S \vert =2$. Write $S = T \cup \{j\}$ and let $i \in T$, so we have
$$c_S(\sum_{S \subseteq V} D_V) = c_T(\sum_{S \subseteq V} D_V)c_{ij}(\sum_{S \subseteq V} D_V).  $$
Note that 
$$ c_T(\sum_{S \subseteq V} D_V) = c_T(\sum_{T \subseteq V} D_V) - \left(\text{terms divisible by some $D_W$ where $W$ contains $T$ but not $j$}\right) $$
and that the first of these two terms vanishes by induction. The second term is killed by multiplication with $ c_{ij}(\sum_{i,j \subseteq V} D_V)$ --- which is zero --- but then also by multiplication with  $ c_{ij}(\sum_{S \subseteq V} D_V)$, by removing terms which necessarily vanish because $W$ and $V$ overlap.

The relations $D_S \cdot c_{S'}(\sum_{S \cup S' \subseteq V} D_V)$ are easily derived: we have $c_{S'}(\sum_{S' \subseteq V} D_V) = 0$ by the previous paragraph. Now multiply with $D_S$ and remove terms which vanish because $S$ and $V$ overlap. \end{proof}

\printbibliography

\end{document}